\documentclass[11pt]{article}
\usepackage[margin=1in]{geometry} 
\usepackage{amssymb,amsfonts,amsmath,bbm,mathrsfs,stmaryrd,mathtools}
\usepackage{xcolor}
\usepackage{url}
\usepackage[T1]{fontenc}
\usepackage{enumerate}

\usepackage{hyperref}
\hypersetup{colorlinks,
             linkcolor=black!75!red,
             citecolor=blue,
             pdftitle={},
             pdfproducer={pdfLaTeX},
             pdfpagemode=None,
             bookmarksopen=true
             bookmarksnumbered=true}

\usepackage{tikz}
\usetikzlibrary{arrows,calc,decorations.pathreplacing,decorations.markings,intersections,shapes.geometric,through,fit,shapes.symbols,positioning,decorations.pathmorphing}

\usepackage{braket}

\usepackage[amsmath,thmmarks,hyperref]{ntheorem}
\usepackage{cleveref}

\creflabelformat{enumi}{#2(#1)#3}

\crefname{section}{Section}{Sections}
\crefformat{section}{#2Section~#1#3} 
\Crefformat{section}{#2Section~#1#3} 

\crefname{subsection}{\S}{\S\S}
\AtBeginDocument{%
  \crefformat{subsection}{#2\S#1#3}%
  \Crefformat{subsection}{#2\S#1#3}%
}

\crefname{subsubsection}{\S}{\S\S}
\AtBeginDocument{%
  \crefformat{subsubsection}{#2\S#1#3}%
  \Crefformat{subsubsection}{#2\S#1#3}%
}

\theoremstyle{plain}

\newtheorem{lemma}{Lemma}[section]
\newtheorem{proposition}[lemma]{Proposition}
\newtheorem{corollary}[lemma]{Corollary}
\newtheorem{theorem}[lemma]{Theorem}

\theoremstyle{nonumberplain}
\newtheorem{theoremN}{Theorem}

\newtheorem{claimN}{Claim}

\theoremstyle{plain}
\theorembodyfont{\upshape}
\theoremsymbol{\ensuremath{\blacklozenge}}

\newtheorem{definition}[lemma]{Definition}

\newtheorem{remark}[lemma]{Remark}

\newtheorem{notation}[lemma]{Notation}

\crefname{definition}{definition}{definitions}
\crefformat{definition}{#2definition~#1#3} 
\Crefformat{definition}{#2Definition~#1#3} 

\crefname{ex}{example}{examples}
\crefformat{example}{#2example~#1#3} 
\Crefformat{example}{#2Example~#1#3} 

\crefname{remark}{remark}{remarks}
\crefformat{remark}{#2remark~#1#3} 
\Crefformat{remark}{#2Remark~#1#3} 

\crefname{convention}{convention}{conventions}
\crefformat{convention}{#2convention~#1#3} 
\Crefformat{convention}{#2Convention~#1#3} 

\crefname{notation}{notation}{notations}
\crefformat{notation}{#2notation~#1#3} 
\Crefformat{notation}{#2Notation~#1#3} 

\crefname{table}{table}{tables}
\crefformat{table}{#2table~#1#3} 
\Crefformat{table}{#2Table~#1#3}

\crefname{lemma}{lemma}{lemmas}
\crefformat{lemma}{#2lemma~#1#3} 
\Crefformat{lemma}{#2Lemma~#1#3} 

\crefname{proposition}{proposition}{propositions}
\crefformat{proposition}{#2proposition~#1#3} 
\Crefformat{proposition}{#2Proposition~#1#3} 

\crefname{corollary}{corollary}{corollaries}
\crefformat{corollary}{#2corollary~#1#3} 
\Crefformat{corollary}{#2Corollary~#1#3} 

\crefname{theorem}{theorem}{theorems}
\crefformat{theorem}{#2theorem~#1#3} 
\Crefformat{theorem}{#2Theorem~#1#3} 

\crefname{enumi}{}{}
\crefformat{enumi}{(#2#1#3)}
\Crefformat{enumi}{(#2#1#3)}

\crefname{assumption}{assumption}{Assumptions}
\crefformat{assumption}{#2assumption~#1#3} 
\Crefformat{assumption}{#2Assumption~#1#3} 

\crefname{equation}{}{}
\crefformat{equation}{(#2#1#3)} 
\Crefformat{equation}{(#2#1#3)}

\numberwithin{equation}{section}
\renewcommand{\theequation}{\thesection-\arabic{equation}}

\theoremstyle{nonumberplain}
\theoremsymbol{\ensuremath{\blacksquare}}

\newtheorem{proof}{Proof}
\newcommand\pf[1]{\newtheorem{#1}{Proof of \Cref{#1}}}

\newcommand\bG{{\mathbb G}}
\newcommand\bH{{\mathbb H}}

\newcommand\bR{{\mathbb R}}

\newcommand\cH{{\mathcal H}}

\newcommand\cT{{\mathcal T}}

\newcommand\fm{{\mathfrak m}}
\newcommand\fn{{\mathfrak n}}

\DeclareMathOperator{\id}{id}

\newcommand\numberthis{\addtocounter{equation}{1}\tag{\theequation}}

\newcommand{\qedhere}{\mbox{}\hfill\ensuremath{\blacksquare}}

\title{Quantum relative modular functions}
\author{Alexandru Chirvasitu}

\begin{document}

\date{}

\newcommand{\Addresses}{{
  \bigskip
  \footnotesize

  \textsc{Department of Mathematics, University at Buffalo, Buffalo,
    NY 14260-2900, USA}\par\nopagebreak \textit{E-mail address}:
  \texttt{achirvas@buffalo.edu}

}}

\maketitle

\begin{abstract}
  Let $\mathbb{H}\trianglelefteq\mathbb{G}$ be a closed normal subgroup of a locally compact quantum group. We introduce a strictly positive group-like element affiliated with $L^{\infty}(\mathbb{G})$ that, roughly, measures the failure of $\mathbb{G}$ to act measure-preservingly on $\mathbb{H}$ by conjugation. The triviality of that element is equivalent to the condition that $\mathbb{G}$ and $\mathbb{G}/\mathbb{H}$ have the same modular element, by analogy with the classical situation. This condition is automatic if $\mathbb{H}\le \mathbb{G}$ is central, and in general implies the unimodularity of $\mathbb{H}$.

  We also describe a bijection between strictly positive group-like elements $\delta$ affiliated with $C_0(\mathbb{G})$ and quantum-group morphisms $\mathbb{G}\to (\mathbb{R},+)$, with the closed image of the morphism easily described in terms of the spectrum of $\delta$. This then implies that property-(T) locally compact quantum groups admit no non-obvious strictly positive group-like elements.
\end{abstract}


\vspace{.5cm}



\section*{Introduction}

The initial motivation for the present paper was the well-known result that nilpotent locally compact groups are unimodular (e.g. \cite[p.318, Corollary 2]{howe-four-nilp}). Several proofs exist in the literature (with more references cited in \Cref{se:relmodel}), but one naive strategy that comes to mind would be as follows: given that nilpotence means that the {\it ascending central series}
\begin{equation*}
  \{1\}\le Z(\bG)\le \cdots\le \bG
\end{equation*}
is finite, perhaps one can employ induction by starting with the abelian group $Z(\bG)$ (which is of course unimodular) and then lifting unimodularity along {\it cocentral} quotients $\bG\to \bG/\bH$ (i.e. quotients by a central closed subgroup). In short, one would need

\begin{claimN}
  A cocentral quotient $\bG/\bH$ of a locally compact group is unimodular if and only if $\bG$ is.
\end{claimN}
Since unimodularity simply means that the {\it modular function} \cite[\S 2.4]{folland} is trivial, this suggests possible generalization:

\begin{claimN}
  For any cocentral quotient $\pi:\bG\to \bG/\bH$, the modular function of $\bG$ is obtained from that of $\bG/\bH$ by restriction along $\pi$.
\end{claimN}

All of this is true and follows easily enough from standard material on modular functions (e.g. from \cite[Theorem 2.51]{folland}), though I have not seen these precise statements. Couched in these terms, though, the statements generalize easily to the framework of locally compact {\it quantum} groups \cite{kvcast,kvwast,kus-univ,wor-mult,sw1,sw2,mnw}, since all of the ingredients are present. To summarize, postponing the notation and terminology until \Cref{se.prel}:
\begin{itemize}
\item For a locally compact quantum group $\bG$ there is a {\it modular element} $\delta_{\bG}$ \cite[\S 7]{kvcast} affiliated with the von Neumann algebra $L^{\infty}(\bG)$, to be thought of as the inverse of the usual modular function (see \Cref{re:howcls} for why this convention is convenient).
\item There are notions of (closed \cite[Definition 2.6]{vs-imp}) normal \cite[Definition 2.10]{vv} and central \cite[Definition 2.3]{kss-cent} quantum subgroups $\bH\le \bG$.
\item As well as quotient groups $\bG/\bH$ by closed normal quantum subgroups \cite[Theorem 2.11]{vv}.
\end{itemize}

All of this allows the formulation of one of the main results below (see \Cref{se:relmodel} and \Cref{cor:unimodquot}):

\begin{theoremN}
  Given a closed central quantum subgroup $\bH\le \bG$ of a locally compact quantum group, the modular elements of $\bG$ and $\bG/\bH$ coincide.

  In particular, a cocentral quotient $\bG/\bH$ is unimodular if and only if $\bG$ is. \qedhere
\end{theoremN}

More generally, there is a very satisfying way of measuring the discrepancy from the previous theorem's conclusion. Summarizing \Cref{th:relmodel,th:normnotcent} and \Cref{pr:relmodel,pr:trivact}:

\begin{theoremN}
  Let $\bH\le \bG$ be a closed normal quantum subgroup. The modular elements $\delta_{\bG}$ and $\delta_{\bG/\bH}$ strongly commute, so their ratio $\delta:=\delta_{\bG} \delta_{\bG/\bH}^{-1}$ is again a strictly positive element affiliated with $L^{\infty}(\bG)$, group-like in the sense that $\Delta_{\bG}(\delta)=\delta\otimes\delta$. 

  That element is trivial precisely when two canonical operator valued weights from $L^{\infty}(\bG)$ to its von Neumann subalgebra $L^{\infty}(\bG/\bH)$ coincide. This condition
  \begin{itemize}
  \item is the quantum analogue of $\bG$ acting measure-preservingly by conjugation on $\bH$;
  \item is automatic when $\bH\le \bG$ is central;
  \item and entails the unimodularity of $\bH$. \qedhere
  \end{itemize}  
\end{theoremN}

The element $\delta$ in the statement above is the relative modular function alluded to in the title of the paper (see \Cref{def:relmodel}): the phrase is meant to indicate that it is relative to an embedding $\bH\le \bG$ rather than absolute, attached to $\bG$ alone.

On a different note but still on the topic of strictly positive group-like elements $\delta$ affiliated with $L^{\infty}(\bG)$, we have (\Cref{pr:mortor}, \Cref{pr:posisgr} and \Cref{th:clim})

\begin{theoremN}
  Every strictly positive group-like element $\delta$ affiliated with $L^{\infty}(\bG)$ induces a quantum-group morphism $\bG\to \bR$ whose closed image is precisely the closed subgroup
  \begin{equation*}
    \{\log t\ |\ 0<t\in\mathrm{Sp}(\delta)\}\subseteq (\bR,+).
  \end{equation*}
  \qedhere
\end{theoremN}

An immediate consequence (\Cref{th:tnodelta} below) is the following generalization of the unimodularity of property-(T) quantum groups \cite[Theorem 6.1]{dds}:

\begin{theoremN}
  If the LCQG $\bG$ has property (T) then the only strictly positive group-like element affiliated with $L^{\infty}(\bG)$ is $1$.  \qedhere
\end{theoremN}

\subsection*{Acknowledgements}

This work is partially supported by NSF grant DMS-2001128. 

I am grateful for insightful comments and pointers to the literature from A. Skalski, P. Kasprzak, P. So{\l}tan, G. Folland and A. Deitmar. 

\section{Preliminaries}\label{se.prel}

Inner products are linear in the second variable, and for vectors $v,w$ in a Hilbert space $(\cH,\braket{-\mid -})$ and an operator $A\in B(\cH)$ we write
\begin{equation*}
  \omega_{v,w}(A):=\braket{v\mid Aw} = \braket{vA\mid w} = \braket{A^*v\mid w}. 
\end{equation*}

Any number of sources cover the needed operator-algebra background: \cite{blk,ped-aut}, etc. Assorted standard notation:
\begin{itemize}
\item $B(\cdot)$ and $K(\cdot)$ denote the algebras of bounded and respectively compact operators on a Hilbert space.
\item $M(\cdot)$ is the {\it multiplier algebra} of a $C^*$-algebra \cite[\S II.7.3]{blk}.
\item For $C^*$-algebras $A$ and $B$ the space $\mathrm{Mor}(A,B)$ of {\it morphisms} from $A$ to $B$ consists (as in \cite[Introduction]{wor-aff}, \cite[\S 1.1]{dkss}, \cite[\S 2]{dsv}, etc.) of those linear, bounded, multiplicative $*$-maps $f:A\to M(B)$ that are {\it non-degenerate} in the sense that $f(A)B$ is norm-dense in $B$.

  We also depict $\pi\in\mathrm{Mor}(A,B)$ as arrows:
  \begin{equation*}
    A \stackrel{\pi}{\rightsquigarrow} B.
  \end{equation*}

\item $M_*$ is the {\it predual} \cite[\S III.2.4]{blk} of a von Neumann algebra, $M_+$ its positive cone (set of positive elements), and $\widehat{M}_+$ its {\it extended positive part} \cite[Definition 1.1]{haag1}.
\item the tensor-product symbol `$\otimes$' has contextual meaning: between $C^*$-algebras it denotes the {\it minimal} (or {\it spatial}) $C^*$ tensor product \cite[\S II.9.1.3]{blk}, between $W^*$-algebras it is the von-Neumann-flavored spatial tensor product of \cite[\S III.5.1.4]{blk}, the Hilbert-space tensor product when appropriate, etc.
\end{itemize}

For the needed material on locally compact quantum groups we refer mainly to \cite{kvcast,kvwast,kus-univ} (with more precise citations below, as needed). The first of these also has an introductory overview of the necessary weight and modular theory; \cite{tak2,strat} are other good sources for this latter topic. Of particular interest are the {\it operator-valued weights} of \cite{haag1,haag2}, covered also in \cite[\S IX.4]{tak2}.

To recall, briefly, the main concept of interest (\cite[Definition 1.1]{kvwast}):
\begin{definition}\label{def:lcqg}
  A locally compact quantum group $\bG$ (occasionally abbreviated LCQG) is a pair $(M,\Delta)$ where
  \begin{itemize}
  \item $M$, denoted also by $L^{\infty}(\bG)$, is a von Neumann algebra.
  \item $\Delta=\Delta_{\bG}$ is a $W^*$ morphism $M\to M\otimes M$, {\it coassociative} in the sense that
    \begin{equation*}
      (\Delta\otimes\id)\circ \Delta = (\id\otimes\Delta)\circ \Delta: M\to M\otimes M\otimes M.
    \end{equation*}
  \item we assume the existence of
    \begin{enumerate}[(a)]
    \item a {\it left Haar weight} on $M$: a normal, semifinite and faithful (n.s.f. for short) weight $\varphi=\varphi_{\bG}$, left-invariant in the sense that
      \begin{equation*}
        \varphi((\omega\otimes\id)\Delta(x)) = \omega(1)\varphi(x)
      \end{equation*}
      for all $\omega\in M_*$ and
      \begin{equation*}
        x\in \fm^+_{\varphi}:=\{x\in M_+\ |\ \varphi(x)<\infty\}. 
      \end{equation*}
    \item similarly, a {\it right} Haar weight $\psi=\psi_{\bG}$, right-invariant:
      \begin{equation*}
        \psi((\id\otimes\omega)\Delta(x)) = \omega(1)\psi(x),\ \forall \omega\in M_*,\ \forall x\in \fm_{\psi}^+.
      \end{equation*}
    \end{enumerate}
  \end{itemize}
\end{definition}

Also central to the discussion is the following object (\cite[Terminology 7.16]{kvcast}).

\begin{definition}
  The {\it modular element} $\delta=\delta_{\bG}$ of an LCQG $\bG$ is the unique (possibly unbounded) operator that is
  \begin{itemize}
  \item {\it strictly positive} in the sense that its spectral resolution \cite[Theorem 13.30]{rud-fa} assigns the zero projection to $\{0\}$ (equivalently: it has dense range \cite[p.841]{kvcast});
  \item {\it affiliated} with $L^{\infty}(\bG)$ in the sense that its spectral projections belong to that von Neumann algebra;
  \end{itemize}
  and such that
  \begin{equation}\label{eq:rvsl}
    \psi_{\bG}(\cdot) = \varphi_{\bG,\delta}(\cdot):=\varphi_{\bG}\left(\delta^{1/2}\cdot\delta^{1/2}\right). 
  \end{equation}
\end{definition}

Other notation pertinent to quantum groups:
\begin{itemize}
\item $L^2(\bG)=L^2(\bG,\varphi_{\bG})$ is the Hilbert space carrying the GNS representation attached to the left Haar weight $\varphi$, equipped with $\Lambda=\Lambda_{\bG}:\fn_{\varphi}\to L^2(\bG)$.
\item $C_0(\bG)\subset L^{\infty}(\bG)$ is the {\it reduced function algebra} of $\bG$, associated to $L^{\infty}(\bG)$ in \cite[\S 1.2]{kvwast} (with the notation $M=L^{\infty}(\bG)$, $M_c=C_0(\bG)$) and studied extensively in \cite{kvcast}.
\item $C^u_0(\bG)$ is the {\it universal} function algebra, constructed in \cite[\S 4]{kus-univ}.
\item $\widehat{\bG}$ is the {\it dual} LCQG: \cite[\S 1.1]{kvwast} for the von Neumann version and \cite[\S 8]{kvcast} for the $C^*$ counterpart.
\item $W\in W_{\bG}$ is the {\it multiplicative unitary} of \cite[Proposition 3.17]{kvcast}: it is defined as an operator on $L^2(\bG)\otimes L^2(\bG)$ by
  \begin{equation*}
    W^*(\Lambda(x)\otimes \Lambda(y)) = \Lambda\otimes\Lambda (\Delta(y)(x\otimes 1)),
  \end{equation*}
  it implements the comultiplication by
  \begin{equation*}
    \Delta(x) = W^*(1\otimes x)W,
  \end{equation*}
  and belongs to
  \begin{equation*}
    M(C_0(\bG)\otimes C_0(\widehat{\bG}))\subset L^{\infty}(\bG)\otimes L^{\infty}(\widehat{\bG}). 
  \end{equation*}
\item $S=S_{\bG}$ and $R=R_{\bG}$ are the {\it antipode} and {\it unitary antipode} of $\bG$ respectively \cite[Terminology 5.42]{kvcast}. 
\end{itemize}

\begin{remark}
  \cite[Proposition 7.10]{kvcast} says that $\delta_{\bG}$ is in fact also affiliated with the $C^*$-algebra $C_0(\bG)$ in the sense of \cite[Definition 1.1]{wor-aff}. Furthermore, it lifts along the surjection $C^u_0(\bG)\to C_0(\bG)$ to a strictly positive element $\delta_{u}=\delta_{u,\bG}$ affiliated with the universal function algebra \cite[Proposition 10.1]{kus-univ}.
\end{remark}

\begin{notation}\label{not:inprime}
  We denote the affiliation relation, in either the $C^*$ or $W^*$ setting, by primed containment symbols: $\in'$ and $\ni'$.
\end{notation}

\subsection{Morphisms}\label{subse:lcqg-mor}

LCQG morphisms have many incarnations; for a review of the theory the reader can consult, for instance, \cite{mrw} (where many of the issues were initially settled), \cite[\S 12]{kus-univ} or \cite[\S 1.3]{dkss}. In particular, attached to such a morphism $\pi:\bH\to \bG$ we have a right action
\begin{equation*}
  \pi_r:L^{\infty}(\bG)\to L^{\infty}(\bG)\otimes L^{\infty}(\bH)
\end{equation*}
as well as a left one,
\begin{equation*}
  \pi_l:L^{\infty}(\bG)\to L^{\infty}(\bG)\otimes L^{\infty}(\bH).
\end{equation*}

Throughout the paper, {\it closed} quantum subgroups $\bH\le \bG$ are as in \cite[Definition 2.6]{vs-imp}, referred to as {\it closed in the sense of Vaes} in \cite[Definition 3.1]{dkss} (to distinguish from a formally weaker version due to Woronowicz): those for which the dual morphism $\widehat{\bG}\to \widehat{\bH}$ corresponds to a comultiplication-intertwining embedding
\begin{equation*}
  L^{\infty}(\widehat{\bH})\subseteq L^{\infty}(\widehat{\bG}). 
\end{equation*}

The {\it centrality} of a quantum subgroup (or more generally, of a morphism) can be cast as the following paraphrase of \cite[Definition 2.3]{kss-cent}:

\begin{definition}
  A morphism $\pi:\bH\to \bG$ is {\it central} if the diagram
  \begin{equation}\label{eq:cent}
    \begin{tikzpicture}[auto,baseline=(current  bounding  box.center)]
      \path[anchor=base] 
      (0,0) node (l) {$L^{\infty}(\bG)$}
      +(4,.5) node (u) {$L^{\infty}(\bG)\otimes L^{\infty}(\bH)$}
      +(4,-.5) node (d) {$L^{\infty}(\bH)\otimes L^{\infty}(\bG)$}
      +(8,0) node (r) {$L^{\infty}(\bG)\otimes L^{\infty}(\bH)$}
      ;
      \draw[->] (l) to[bend left=6] node[pos=.5,auto] {$\scriptstyle \pi_r$} (u);
      \draw[->] (u) to[bend left=6] node[pos=.5,auto] {$\scriptstyle \id$} (r);
      \draw[->] (l) to[bend right=6] node[pos=.5,auto,swap] {$\scriptstyle \pi_l$} (d);
      \draw[->] (d) to[bend right=6] node[pos=.5,auto,swap] {$\scriptstyle \text{flip}$} (r);
    \end{tikzpicture}
  \end{equation}
commutes. 
\end{definition}

For a closed quantum subgroup $\iota:\bH\le \bG$ one can define the left and right quantum homogeneous $\bG$-spaces (e.g. \cite[Definition 4.1]{vs-imp}):
\begin{align*}
  L^{\infty}(\bG/\bH)&:=\{x\in L^{\infty}(\bG)\ |\ \iota_r(x) = x\otimes 1\}\numberthis\label{eq:homog}\\
  L^{\infty}(\bH\backslash \bG)&:=\{x\in L^{\infty}(\bG)\ |\ \iota_l(x) = 1\otimes x\}.\\
\end{align*}

Morphisms of locally compact quantum groups preserve unitary antipodes; this is well known, but we set out the claim here in precisely the form needed below (see e.g. \cite[equation (2.2b)]{ks-proj}).

\begin{lemma}\label{le:rinv}
  For an LCQG morphism $\pi:\bH\to \bG$ the diagrams
  \begin{equation*}
    \begin{tikzpicture}[auto,baseline=(current  bounding  box.center)]
      \path[anchor=base] 
      (0,0) node (l) {$L^{\infty}(\bG)$}
      +(5,.5) node (u) {$L^{\infty}(\bG)\otimes L^{\infty}(\bH)$}
      +(2,-.5) node (dl) {$L^{\infty}(\bG)$}
      +(6,-.5) node (d) {$L^{\infty}(\bH)\otimes L^{\infty}(\bG)$}
      +(10,0) node (r) {$L^{\infty}(\bG)\otimes L^{\infty}(\bH)$}
      ;
      \draw[->] (l) to[bend left=6] node[pos=.5,auto] {$\scriptstyle \pi_r$} (u);
      \draw[->] (u) to[bend left=6] node[pos=.5,auto] {$\scriptstyle R_{\bG}\otimes R_{\bH}$} (r);
      \draw[->] (l) to[bend right=6] node[pos=.5,auto,swap] {$\scriptstyle R_{\bG}$} (dl);
      \draw[->] (dl) to[bend right=6] node[pos=.5,auto,swap] {$\scriptstyle \pi_l$} (d);
      \draw[->] (d) to[bend right=6] node[pos=.5,auto,swap] {$\scriptstyle \text{flip}$} (r);
    \end{tikzpicture}
  \end{equation*}
  and 
  \begin{equation*}
    \begin{tikzpicture}[auto,baseline=(current  bounding  box.center)]
      \path[anchor=base] 
      (0,0) node (l) {$L^{\infty}(\bG)$}
      +(5,.5) node (u) {$L^{\infty}(\bH) \otimes L^{\infty}(\bG)$}
      +(2,-.5) node (dl) {$L^{\infty}(\bG)$}
      +(6,-.5) node (d) {$L^{\infty}(\bG)\otimes L^{\infty}(\bH)$}
      +(10,0) node (r) {$L^{\infty}(\bH) \otimes L^{\infty}(\bG)$}
      ;
      \draw[->] (l) to[bend left=6] node[pos=.5,auto] {$\scriptstyle \pi_l$} (u);
      \draw[->] (u) to[bend left=6] node[pos=.5,auto] {$\scriptstyle R_{\bH} \otimes R_{\bG} $} (r);
      \draw[->] (l) to[bend right=6] node[pos=.5,auto,swap] {$\scriptstyle R_{\bG}$} (dl);
      \draw[->] (dl) to[bend right=6] node[pos=.5,auto,swap] {$\scriptstyle \pi_r$} (d);
      \draw[->] (d) to[bend right=6] node[pos=.5,auto,swap] {$\scriptstyle \text{flip}$} (r);
    \end{tikzpicture}
  \end{equation*}
  commute.
\end{lemma}
\begin{proof}
  This follows from \cite[Theorems 5.3 and 5.5]{mrw}, which describe $\pi_r$ and $\pi_l$ in terms of a single object attached to the morphism $\pi$ (a {\it bicharacter}, in the language of \cite[\S 3]{mrw}).
\end{proof}

In the discussion below, we follow \cite{kvcast} in denoting by
\begin{itemize}
\item $\sigma_t$ (or $\sigma_{\bG,t}$ when wishing to emphasize the group) the modular automorphism group of a left Haar weight \cite[\S 1.3]{kvcast};
\item $\sigma'_t=\sigma'_{\bG,t}$ the modular group of a right Haar weight \cite[p.846]{kvcast};
\item $\tau_t=\tau_{\bG,t}$ the scaling group of $\bG$ \cite[Terminology 5.42]{kvcast} and by $\nu=\nu_{\bG}$ its {\it scaling constant} \cite[Terminology 7.16]{kvcast}. 
\end{itemize}

In addition to the antipode-intertwining properties noted in \Cref{le:rinv}, it will also be useful to record the compatibility between $\pi_{l,r}$ and these one-parameter groups.

\begin{lemma}\label{le:act1par}
  For an LCQG morphism $\pi:\bH\to \bG$ we have
  \begin{equation}\label{eq:pils}
    \pi_l\sigma_{\bG,t} = (\tau_{\bH,t}\otimes \sigma_{\bG,t})\pi_l:L^{\infty}(\bG)\to L^{\infty}(\bH)\otimes L^{\infty}(\bG).
  \end{equation}
  \begin{equation}\label{eq:pirs}
    \pi_r\sigma'_{\bG,t} = (\sigma'_{\bG,t}\otimes \tau_{\bH,-t})\pi_r:L^{\infty}(\bG)\to L^{\infty}(\bG)\otimes L^{\infty}(\bH).
  \end{equation}
  \begin{equation}\label{eq:pilt}
    \pi_l\tau_{\bG,t} = (\tau_{\bH,t}\otimes \tau_{\bG,t})\pi_l:L^{\infty}(\bG)\to L^{\infty}(\bH)\otimes L^{\infty}(\bG).
  \end{equation}
  \begin{equation}\label{eq:pirt}
    \pi_r\tau_{\bG,t} = (\tau_{\bG,t}\otimes \tau_{\bH,t})\pi_r:L^{\infty}(\bG)\to L^{\infty}(\bG)\otimes L^{\infty}(\bH).
  \end{equation}
\end{lemma}
\begin{proof}
  The style of proof is the same for all of these, so we focus on \Cref{eq:pils}.

  All three one-parameter groups lift to the universal quantum-group function algebras $C^u_0(\bG)$ (and analogue for $\bH$) of \cite{kus-univ}: see \cite[\S 8]{kus-univ} for the modular groups $\sigma$ and $\sigma'$ and \cite[\S 9]{kus-univ} for $\tau$.

  At the universal level we have \cite[Proposition 9.2]{kus-univ}
  \begin{equation}\label{eq:delsuniv}
    \Delta_{\bG} \sigma^u_{\bG,t} = (\tau^u_{\bG,t}\otimes \sigma^u_{\bG,t})\Delta_{\bG}:C^u_0(\bG)\to M(C^u_0(\bG)\otimes C^u_0(\bG)). 
  \end{equation}
  Now apply the universal incarnation
  \begin{equation*}
    \pi^u:C^u_0(\bG)\to M(C_0(\bH)^u)
  \end{equation*}
  of $\pi$ (\cite[\S 4]{mrw}, \cite[\S 12]{kus-univ}) to the left leg of \Cref{eq:delsuniv} to obtain 
  \begin{equation*}
    \pi^u_l \sigma^u_{\bG,t} = (\pi^u\tau^u_{\bG,t}\otimes \sigma^u_{\bG,t})\Delta_{\bG}:C^u_0(\bG)\to M(C_0(\bH)^u\otimes C^u_0(\bG)),
  \end{equation*}
  where
  \begin{equation*}
    \pi^u_l:=(\pi^u\otimes\id)\Delta_{\bG}.
  \end{equation*}
  Next, use the scaling-group-intertwining property
  \begin{equation*}
    \pi^u\tau^u_{\bG,t}=\tau^u_{\bH,t}\pi^u
  \end{equation*}
  of $\pi^u$ (which follows, for instance, from \cite[Proposition 3.10]{mrw}) on the right-hand side to produce
  \begin{align*}
    \pi^u_l \sigma^u_{\bG,t} &= (\tau^u_{\bH,t} \pi^u\otimes \sigma^u_{\bG,t})\Delta_{\bG}\\
                             &=(\tau^u_{\bH,t} \otimes \sigma^u_{\bG,t})\pi_l^u.
  \end{align*}
  Finally, to conclude, note that this reduces precisely to the desired identity \Cref{eq:pils}, because
  \begin{equation*}
    \begin{tikzpicture}[auto,baseline=(current  bounding  box.center)]
      \path[anchor=base] 
      (0,0) node (lu) {$C^u_0(\bG)$}
      +(4,0) node (ru) {$M(C_0(\bH)^u\otimes C^u_0(\bG))$}
      +(0,-1) node (ld) {$L^{\infty}(\bG)$}
      +(4,-1) node (rd) {$L^{\infty}(\bH)\otimes L^{\infty}(\bG)$}
      ;
      \draw[->] (lu) to[bend left=0] node[pos=.5,auto] {$\scriptstyle \pi^u_l$} (ru);
      \draw[->] (ld) to[bend left=0] node[pos=.5,auto,swap] {$\scriptstyle \pi_l$} (rd);
      \draw[->] (lu) to[bend left=0] node[pos=.5,auto] {$\scriptstyle $} (ld);
      \draw[->] (ru) to[bend left=0] node[pos=.5,auto] {$\scriptstyle $} (rd);
    \end{tikzpicture}
  \end{equation*}
commutes \cite[Proposition 12.1]{kus-univ}. 
\end{proof}

An immediate consequence of \Cref{le:act1par} and the definitions of the quantum homogeneous spaces $\bG/\bH$ and $\bH\backslash \bG$:

\begin{corollary}\label{cor:homoginv}
  For any closed locally compact quantum subgroup $\bH\le \bG$
  \begin{enumerate}[(1)]
  \item $L^{\infty}(\bG/\bH)\subseteq L^{\infty}(\bG)$ is invariant under $\tau_{\bG,t}$ and $\sigma'_{\bG,t}$;
  \item and similarly, $L^{\infty}(\bH\backslash \bG)\subseteq L^{\infty}(\bG)$ is invariant under $\tau_{\bG,t}$ and $\sigma_{\bG,t}$. \qedhere
  \end{enumerate}
\end{corollary}

\section{Relative modular elements}\label{se:relmodel}

One of the main results of this section (to be strengthened later, when more language has been introduced) is 

\begin{theorem}\label{th:resmod}
  Let $\bH\le \bG$ be a closed central subgroup of a locally compact quantum group. The modular element of $\bG$ coincides with that of $\bG/\bH$. 
\end{theorem}

\begin{remark}\label{re:howcls}
  To put \Cref{th:resmod} into some perspective, with centrality being the last of a series of progressively more stringent conditions, note that
  \begin{itemize}
  \item If $\bH\trianglelefteq \bG$ is a closed normal subgroup then the modular element $\delta_{\bG}$ of $\bG$ {\it restricts to} $\delta_{\bH}$ in the sense of \cite[Definition 3.3]{bcv} (by \cite[Theorem 3.4 and Corollary 3.9]{bcv}).

    The restriction terminology employed there is chosen so that classically it specializes back to what one would guess. The modular function of a locally compact group $\bG$ is typically denoted by $\Delta_{\bG}$ or plain $\Delta$ (\cite[\S A.3]{bdv}, \cite[\S 1.4]{de}, \cite[\S 2.4]{folland}, etc.). Here, in order to avoid confusion with the comultiplication, we write
    \begin{equation*}
      \delta_{\bG}(x):=\Delta(x)^{-1},\ x\in \bG. 
    \end{equation*}
    This is compatible with the previous use of the symbol $\delta$, in the general context of quantum groups: on the one hand we have the relation \Cref{eq:rvsl} between left and right Haar weights, while on the other hand, classically, we have
    \begin{equation*}
      d\mu_{right}(x) = \Delta(x)^{-1}d\mu_{left}(x)
    \end{equation*}
    by \cite[Proposition 2.31]{folland}.    

    As the name suggests, then, $\delta_{\bG}$ restricting to $\delta_{\bH}$ as in \cite[Definition 3.3]{bcv} means precisely that $\delta_{\bH}=\delta_{\bG}|_{\bH}$ for ordinary locally compact groups.
    
  \item If furthermore $\bH\trianglelefteq \bG$ is unimodular, it follows that the modular function {\it factors through $\bG\to \bG/\bH$}, in the sense that
    \begin{equation*}
      \delta_{\bG}^{it}\in L^{\infty}(\bG/\bH),\ \forall t\in \bR;
    \end{equation*}
    in other words, $\delta$ is affiliated with the von Neumann subalgebra $L^{\infty}(\bG/\bH)\subseteq L^{\infty}(\bG)$. This follows from \cite[Theorem 3.4, condition (2)]{bcv} and classically it means that the morphism
    \begin{equation*}
      \delta_{\bG}:\bG\to (\bR^{\times},\cdot)
    \end{equation*}
    factors through $\bG/\bH$. 
  \item Finally, it takes centrality to ensure that that factorization in fact coincides with the modular function
    \begin{equation*}
      \delta_{\bG/\bH}:\bG/\bH\to (\bR^{\times},\cdot). 
    \end{equation*}
  \end{itemize}
\end{remark}

Before moving on to the proof of \Cref{th:resmod}, note the following immediate consequence.

\begin{corollary}\label{cor:unimodquot}
  If $\bH\le \bG$ is a closed central subgroup of a locally compact quantum group then $\bG$ is unimodular if and only if $\bG/\bH$ is. \qedhere
\end{corollary}

As yet another consequence, we have the unimodularity of nilpotent locally compact (classical) groups. The result is well known, but the proofs one encounters tend to be different in flavor: \cite[Corollary 2, p.318]{howe-four-nilp} leverages some structure results on nilpotent groups, while \cite[Example A.3.7]{bdv} uses (via \cite[Exercise A.8.10]{bdv}) the fact that nilpotent groups have {\it subexponential growth}.

\begin{corollary}
  Nilpotent locally compact groups are unimodular. 
\end{corollary}
\begin{proof}
  Filter the nilpotent group $\bG$ with its ascending central series
  \begin{equation*}
    \{1\}\le Z(\bG)\le \cdots\le \bG
  \end{equation*}
  (finite, by the nilpotence assumption), and proceed by induction on the length of that series: the base case of abelian groups is trivial, and the induction step passes from a quotient to a central extension using \Cref{cor:unimodquot}.
\end{proof}

For a closed quantum subgroup $\bH\le \bG$ we will work with the two operator-valued weights $\cT_{l}$ and $\cT_{r}$ defined by
\begin{equation}\label{eq:tr}
  \begin{tikzpicture}[auto,baseline=(current  bounding  box.center)]
    \path[anchor=base] 
    (0,0) node (l) {$L^{\infty}(\bG)$}
    +(4,.5) node (u) {$L^{\infty}(\bG)\otimes L^{\infty}(\bH)$}
    +(4,-.5) node (d) {$\widehat{L^{\infty}(\bG/\bH)}_+$}
    +(8,0) node (r) {$\widehat{L^{\infty}(\bG)}_+,$}
    ;
    \draw[->] (l) to[bend left=6] node[pos=.5,auto] {$\scriptstyle \pi_r$} (u);
    \draw[->] (u) to[bend left=6] node[pos=.5,auto] {$\scriptstyle \id\otimes \varphi_{\bH}$} (r);
    \draw[->] (l) to[bend right=6] node[pos=.5,auto,swap] {$\scriptstyle \cT_{l}$} (d);
    \draw[->] (d) to[bend right=6] node[pos=.5,auto,swap] {$\scriptstyle \subseteq$} (r);
  \end{tikzpicture}
\end{equation}
with the `$l$' subscript indicating left invariance or mapping to the {\it left} coset space, and similarly,
\begin{equation}\label{eq:tl}
  \begin{tikzpicture}[auto,baseline=(current  bounding  box.center)]
    \path[anchor=base] 
    (0,0) node (l) {$L^{\infty}(\bG)$}
    +(4,.5) node (u) {$L^{\infty}(\bH)\otimes L^{\infty}(\bG)$}
    +(4,-.5) node (d) {$\widehat{L^{\infty}(\bH\backslash \bG)}_+$}
    +(8,0) node (r) {$\widehat{L^{\infty}(\bG)}_+$.}
    ;
    \draw[->] (l) to[bend left=6] node[pos=.5,auto] {$\scriptstyle \pi_l$} (u);
    \draw[->] (u) to[bend left=6] node[pos=.5,auto] {$\scriptstyle \psi_{\bH}\otimes\id$} (r);
    \draw[->] (l) to[bend right=6] node[pos=.5,auto,swap] {$\scriptstyle \cT_{r}$} (d);
    \draw[->] (d) to[bend right=6] node[pos=.5,auto,swap] {$\scriptstyle \subseteq$} (r);
  \end{tikzpicture}
\end{equation}

\begin{lemma}\label{le:intert-gen}
  For a closed quantum subgroup $\bH\le \bG$ of a closed quantum subgroup we have
  \begin{equation*}
    R_{\bG}\circ \cT_{l} = \cT_{r}\circ R_{\bG}\quad\text{and}\quad R_{\bG}\circ \cT_{r} = \cT_{l}\circ R_{\bG}.
  \end{equation*}
\end{lemma}
\begin{proof}
  That $R_{\bG}$ interchanges $\widehat{L^{\infty}(\bG/\bH)}_+$ and $\widehat{L^{\infty}(\bH\backslash \bG)}_+$ follows from \Cref{le:rinv} (applied to the embedding morphism $\iota:\bH\le \bG$) and the definition \Cref{eq:homog} of the two quantum homogeneous spaces (see also \cite[Proposition 3.3]{ks-proj}).
  
  As for the substance of the statement, it too is an immediate consequence of \Cref{le:rinv}: to obtain $R_{\bG}\circ \cT_{l} = \cT_{r}\circ R_{\bG}$, for instance, apply $\psi_{\bH}$ to the right-hand leg of the top diagram in \Cref{le:rinv} and use the fact that (by definition!) $\varphi_{\bH}$ is nothing but $\psi_{\bH}\circ R_{\bH}$. The other equation follows similarly from the second diagram.
\end{proof}

For a closed {\it normal} quantum subgroup $\bH\le \bG$ the two homogeneous spaces coincide (and this in fact characterizes normality; \cite[\S 4]{ks-proj}, \cite[Theorem 2.11]{vv}):
\begin{equation*}
  \bH\text{ normal}\Longleftrightarrow L^{\infty}(\bG/\bH) = L^{\infty}(\bH\backslash \bG).
\end{equation*}
In that case $\bG/\bH$ is an LCQG in its own right and $R_{\bG}$ restricts to $R_{\bG/\bH}$. \Cref{le:intert-gen} thus implies

\begin{lemma}\label{le:intert-norm}
  For a closed normal quantum subgroup $\bH\trianglelefteq \bG$ of a closed quantum subgroup we have
  \begin{equation*}
    R_{\bG/\bH}\circ \cT_{l} = \cT_{r}\circ R_{\bG}\quad\text{and}\quad R_{\bG/\bH}\circ \cT_{r} = \cT_{l}\circ R_{\bG}.
  \end{equation*}
\end{lemma}

Recall \cite[Proposition]{chk} also that for closed normal quantum subgroups we have a Weyl-type ``disintegration formula''
\begin{equation}\label{eq:weyl}
  \varphi_{\bG} = \varphi_{\bG/\bH}\circ \cT_{l}.
\end{equation}
Naturally, since left Haar weights are only determined up to positive scaling, the content of this claim is that the right-hand side of \Cref{eq:weyl} is left-invariant. Having fixed a left Haar weight $\varphi$ though, we are making the convention that the corresponding right Haar weight $\psi$ is determined by it: $\psi=\varphi\circ R$. The following observation says that this switch from left to right Haar weights is compatible with the operator-valued weights $\cT$.

\begin{lemma}\label{le:lrhaar}
  For a closed, normal quantum subgroup $\bH\trianglelefteq \bG$ of a locally compact quantum group we have
  \begin{equation}\label{eq:simultscale}
    \varphi_{\bG} = \varphi_{\bG/\bH}\circ \cT_{l} \Longleftrightarrow \psi_{\bG} = \psi_{\bG/\bH}\circ \cT_{r}.
  \end{equation}
\end{lemma}
\begin{proof}
  This follows from the various intertwining properties of the unitary antipode(s), already noted above: suppose we have scaled the left Haar weights so that the left hand equation holds. We then have
  \begin{align*}
    \psi_{\bG} &= \varphi_{\bG}\circ R_{\bG} \quad \text{by convention}\\
               &=\varphi_{\bG/\bH}\circ \cT_{l}\circ R_{\bG}\quad \text{by assumption}\\
               &=\varphi_{\bG/\bH}\circ R_{\bG/\bH}\circ \cT_{r}\quad \text{\Cref{le:intert-norm}}\\
               &=\psi_{\bG/\bH}\circ \cT_{r}\quad \text{again by convention}.
  \end{align*}
  This concludes the proof. 
\end{proof}

\pf{th:resmod}
\begin{th:resmod}
  Under the centrality assumption $\bH$ will in particular be {\it abelian} (in the sense that $L^{\infty}(\bH)$ is cocommutative) and hence unimodular, so its left and right Haar weights coincide: $\varphi_{\bH}=\psi_{\bH}$. $\bH$ is furthermore normal so that
  \begin{equation*}
    L^{\infty}(\bH\backslash \bG) = L^{\infty}(\bG/\bH),
  \end{equation*}
  and the two operator-valued weights $\cT_{l}$ and $\cT_{r}$ introduced in \Cref{eq:tr} and \Cref{eq:tl} coincide:
  \begin{equation*}
    \cT:=\cT_{l}=\cT_{r}. 
  \end{equation*}
  According to \Cref{le:lrhaar} we can scale the various Haar weights so that 
  \begin{equation}\label{eq:phipsires}
    \varphi_{\bG/\bH}\circ \cT = \varphi_{\bG} \quad\text{and}\quad \psi_{\bG/\bH}\circ \cT = \psi_{\bG}.
  \end{equation}

  We have
  \begin{align*}
    \nu_{\bG}^{\frac 12 it^2}\delta_{\bG}^{it}&= (D\psi_G:D\varphi_{\bG})_t\quad\text{\cite[Proposition 4.4]{vaes-rn} and \cite[Proposition 7.12 (6)]{kvcast}}\\
                                                  &=(D\psi_{\bG/\bH}\circ\cT:D\varphi_{\bG/\bH}\circ\cT)_t\quad\text{by \Cref{eq:phipsires}}\\
                                                  &=(D\psi_{\bG/\bH}:D\varphi_{\bG/\bH})_t\quad\text{\cite[Theorem 4.7]{haag1}}\\
                                              &=\nu_{\bG/\bH}^{\frac 12 it^2}\delta_{\bG/\bH}^{it}\quad\text{analogous to the first equality}.\\
  \end{align*}
  This is already sufficient to draw the desired conclusion 
  \begin{equation*}
    \delta_{\bG}^{it} = \delta_{\bG/\bH}^{it},
  \end{equation*}
  since given a positive real $\lambda$ and a positive (possibly unbounded) operator $\delta$, the latter can be recovered from $u_t:=\lambda^{it^2}\delta^{it}$: the logarithm $\log \delta$ (obtained by applying $\log$ to the positive operator $\delta$ as usual, via functional calculus \cite[Theorem 13.24]{rud-fa}) can be obtained \cite[\S A.3]{tak2} as
  \begin{equation*}
    i\log\delta\xi  = \lim_{t\to 0}\frac{\delta^{it}-1}t \xi = \lim_{t\to 0}\frac{u_t-1}t \xi
  \end{equation*}
  for $\xi$ ranging over a dense subspace of the ambient Hilbert space.
\end{th:resmod}


We also record the following remark, obtained in passing in the course of the above proof.

\begin{corollary}
  If $\bH\le \bG$ is a central, closed, normal quantum subgroup of a locally compact quantum group the scaling constants of $\bG$ and $\bG/\bH$ coincide.
\end{corollary}
\begin{proof}
  The proof of \Cref{th:resmod} actually shows that
  \begin{equation*}
    \nu_{\bG}^{\frac 12 it^2}\delta_{\bG}^{it} = \nu_{\bG/\bH}^{\frac 12 it^2}\delta_{\bG/\bH}^{it},\ \forall t\in \bR
  \end{equation*}
  and then concludes that the $\delta$ factors coincide: $\delta_{\bG}^{it}=\delta_{\bG/\bH}^{it}$. The $\nu$ factors must thus also coincide:
  \begin{equation*}
    \nu_{\bG}^{\frac 12 it^2} = \nu_{\bG/\bH}^{\frac 12 it^2},\ \forall t\in \bR,
  \end{equation*}
  which of course implies $\nu_{\bG}=\nu_{\bG/\bH}$.
\end{proof}

\begin{remark}\label{re:lie}  
  By way of bolstering the intuitive plausibility of \Cref{th:resmod}, it might be instructive to consider the classical setup whereby $\bG$ is a connected Lie group. In that case we know \cite[Proposition 2.30]{folland} that
  \begin{equation*}
    \delta_{\bG}(x) = \Delta_{\bG}(x)^{-1} = \det Ad(x),
  \end{equation*}
  where $Ad:\bG\to GL(Lie(\bG))$ is the adjoint action. Choose a decomposition
  \begin{equation*}
    Lie(\bG) = Lie(\bH)\oplus V
  \end{equation*}
  and a compatible basis that will give matrix expressions for adjoint-action operators. The centrality of $\bH$ then ensures that
  \begin{equation}\label{eq:block}
    Ad(x) =
    \begin{pmatrix}
      I & *\\
      0 & Ad(\overline{x})
    \end{pmatrix},
  \end{equation}
  where
  \begin{equation*}
    \bG\ni x\mapsto \overline{x}\in \bG/\bH. 
  \end{equation*}
  Plainly, the determinant of \Cref{eq:block} equals that of its lower right-hand block, hence \Cref{th:resmod} in this case.
\end{remark}

\Cref{re:lie} also suggests what is needed in order to extend \Cref{th:resmod} to {\it normal} (non-central) closed quantum subgroups. In that case, \Cref{eq:block} takes the form
\begin{equation}\label{eq:blockgen}
  Ad(x) =
  \begin{pmatrix}
    Ad(x|_{\bH}) & *\\
    0 & Ad(\overline{x})
  \end{pmatrix},
\end{equation}
where $Ad(x|_{\bH})$ denotes the adjoint action by $x$ on $Lie(\bH)$. Taking determinants we thus have
\begin{equation}\label{eq:normnotcent-cls}
  \delta_{\bG}(x) = \det Ad(x) = \det Ad(x|_{\bH}) \cdot \det Ad(\overline{x}) = \det Ad(x|_{\bH}) \cdot \delta_{\bG/\bH}(\overline{x}). 
\end{equation}
The ``correction factor'' away from \Cref{th:resmod} is thus $\det Ad(x|_{\bH}^{-1})$. Its quantum counterpart, for normal $\bH\trianglelefteq \bG$, will be a measure of how far apart the two operator-valued weights
\begin{equation*}
  \cT_{l},\ \cT_{r}:L^{\infty}(\bG)\to \widehat{L^{\infty}(\bG/\bH)}_+
\end{equation*}
are from each other: it was their coincidence that captured the triviality of the upper left-hand block in \Cref{eq:block}. Measuring this discrepancy between $\cT_{r}$ and $\cT_{l}$ is precisely what the Radon-Nikodym derivative $(D\cT_{r}:D\cT_{l})_t$ of \cite[Definition 6.2]{haag2} is designed to do, so that construction features below.

As \cite[Proposition 5.5]{vaes-rn} makes clear, such Radon-Nikodym derivatives ought to be intimately related to how one of the operator-valued weights $\cT_{l,r}$ evolves under the modular group of the other. The following result examines this.

\begin{lemma}\label{le:tsigma}
  For a closed locally compact quantum group $\bH\le \bG$ we have
  \begin{equation}\label{eq:tls}
    \cT_l\sigma'_{\bG,t} = \nu_{\bH}^t\sigma'_{\bG,t}\cT_l
  \end{equation}
  and similarly,
  \begin{equation}\label{eq:trs}
    \cT_r\sigma_{\bG,t} = \nu_{\bH}^{-t}\sigma_{\bG,t}\cT_r
  \end{equation}
\end{lemma}
\begin{proof}
  Note first that the right-hand sides actually make sense: by \Cref{cor:homoginv} the modular group $\sigma'_{\bG,t}$ leaves the codomain
  \begin{equation*}
    \widehat{L^{\infty}(\bG/\bH)}_+ \subseteq \widehat{L^{\infty}(\bG)}_+
  \end{equation*}
  of $\cT_l$ invariant, and similarly for $\cT_r$. The two arguments being entirely parallel, we only run through the first. Denoting by $\pi:\bH\to \bG$ the embedding:
  \begin{align*}
    \cT_l\sigma'_{\bG,t} &= (\id\otimes\varphi_{\bH})\pi_r\sigma'_{\bG,t}\quad\text{by definition}\\
                         &= (\id\otimes\varphi_{\bH})(\sigma'_{\bG,t}\otimes \tau_{\bH,-t})\pi_r\quad \text{\Cref{eq:pirs}}\\
                         &=\nu_{\bH}^t (\sigma'_{\bG,t}\otimes \varphi_{\bH})\pi_r\quad \text{\cite[Proposition 6.8 (3)]{kvcast}}\\
                         &=\nu_{\bH}^t\sigma'_{\bG,t}\cT_l\quad \text{by the definition of $\cT_l$ again}.
  \end{align*}
  This concludes the proof of \Cref{eq:tls}. 
\end{proof}

Note, in passing, that for normal quantum subgroups Weyl disintegration transports over to scaling constants.

\begin{proposition}\label{pr:scalemult}
  For a closed, normal quantum subgroup $\bH\trianglelefteq \bG$ of a locally compact quantum group we have
  \begin{equation*}
    \nu_{\bG} = \nu_{\bH}\nu_{\bG/\bH}. 
  \end{equation*}
\end{proposition}
\begin{proof}
  Throughout the proof we assume we have fixed Haar weights on $\bG$ and $\bG/\bH$ so that both conditions in \Cref{eq:simultscale} hold (as that result says we may):
  \begin{equation}\label{eq:phipsi}
    \varphi_{\bG/\bH}\circ \cT_l = \varphi_{\bG}\quad\text{and}\quad \psi_{\bG/\bH}\circ \cT_r = \psi_{\bG}.
  \end{equation}
  By definition (\cite[Proposition 6.8 and Terminology 7.16]{kvcast}), $\nu_{\bG/\bH}$ can be expressed by
  \begin{equation}\label{eq:nugh}
    \varphi_{\bG/\bH}\circ \sigma'_{\bG/\bH,t} = \nu_{\bG/\bH}^t \varphi_{\bG/\bH}.
  \end{equation}
  Now precompose both sides with $\cT_l$:
  \begin{align*}
    \nu_{\bG/\bH}^t\varphi_{\bG} &= \nu_{\bG/\bH}^t\varphi_{\bG/\bH}\cT_l\quad \text{\Cref{eq:phipsi}}\\
                                 &= \varphi_{\bG/\bH}\circ \sigma'_{\bG/\bH,t}\cT_l\quad\text{\Cref{eq:nugh}}\\
                                 &=\varphi_{\bG/\bH}\circ \sigma'_{\bG,t}\cT_l\quad\text{\cite[Theorem 4.7]{haag1}}\\
                                 &=\nu_{\bH}^{-t}\varphi_{\bG/\bH}\circ \cT_l\sigma'_{\bG,t}\quad\text{\Cref{eq:tls}}\\
                                 &=\nu_{\bH}^{-t}\varphi_{\bG}\sigma'_{\bG,t}\quad\text{\Cref{eq:phipsi} again}\\
                                 &=\nu_{\bH}^{-t}\nu_{\bG}^t\varphi_{\bG}\quad\text{\cite[Proposition 6.8 (3)]{kvcast}}. 
  \end{align*}
  This gives the desired result $\nu_{\bG/\bH} = \nu_{\bH}^{-1}\nu_{\bG}$.
\end{proof}

In light of \cite[Proposition 5.5]{vaes-rn}, \Cref{le:tsigma} is strongly suggestive of \Cref{th:relmodel} below. In the statement, we refer to the {\it modular group} $\sigma^{\cT}_{\bG,t}$ of an operator-valued weight $\cT$ on $L^{\infty}(\bG)$; recall that for an operator-valued weight $\cT:M\to \widehat{N}_+$ that modular group is
\begin{equation*}
  \sigma^{\cT}_t := \sigma^{\theta\circ\cT}_t|_{N^c},\ \theta\text{ any n.s.f. weight on }N:
\end{equation*}
this is \cite[Definition 6.2 (1)]{haag2}, relying on the fact that by \cite[Proposition 6.1 (1)]{haag2} the definition does not depend on $\theta$.

\begin{theorem}\label{th:relmodel}
  Let $\bH\trianglelefteq \bG$ be a closed, normal locally compact quantum subgroup. There is a strictly positive element $\delta=\delta_{\bG\triangleright \bH}$ affiliated with the relative commutant
  \begin{equation*}
    L^{\infty}(\bG/\bH)^c = L^{\infty}(\bG)\cap L^{\infty}(\bG/\bH)'
  \end{equation*}
  such that
  \begin{equation}\label{eq:dtdt}
    (D\cT_r:D\cT_l)_t = \nu_{\bH}^{\frac 12 it^2}\delta_{\bG\triangleright H}^{it}
  \end{equation}
  and
  \begin{equation}\label{eq:sigmatdel}
    \sigma^{\cT_l}_{\bG,s}(\delta^{it}) = \nu_{\bH}^{ist}\delta^{it},\ \forall s,t\in \bR. 
  \end{equation}
\end{theorem}
\begin{proof}
  By \cite[Definition 6.2]{haag2}, the Radon-Nikodym derivative $(D\cT':D\cT)_t$ between two operator-valued weights is simply $(D\omega \cT':D\omega\cT)_t$ for any n.s.f. weight $\omega$ (since that derivative does not depend on $\omega$ \cite[Proposition 6.1]{haag2}). We are thus free to choose the weight conveniently:
  \begin{equation*}
    (D\cT_r:D\cT_l)_t = (D\varphi_{\bG/\bH}\circ\cT_r:D\varphi_{\bG/\bH}\circ\cT_l)_t,\ t\in \bR. 
  \end{equation*}
  Assuming \Cref{eq:phipsi} (as we will), the right-hand weight is nothing but $\varphi_{\bG}$, and its modular group is $\sigma_{\bG,t}$. Under that group, the other weight evolves as follows:
  \begin{align*}
    \varphi_{\bG/\bH}\circ\cT_r\sigma_{\bG,t} &= \nu_{\bH}^{-t}\varphi_{\bG/\bH}\circ\sigma_{\bG,t}\cT_r \quad\text{\Cref{eq:trs}}\\
                                              &= \nu_{\bH}^{-t}\varphi_{\bG/\bH}\circ\sigma_{\bG/\bH,t}\cT_r\quad\text{\cite[Theorem 4.7]{haag1}}\\
                                              &=\nu_{\bH}^{-t}\varphi_{\bG/\bH}\circ\cT_r\quad\text{($\varphi_{\bG/\bH}$ invariant under its own modular group)}.
  \end{align*}
  Now \cite[Proposition 5.5, (ii) $\Rightarrow$ (iv)]{vaes-rn} shows that
  \begin{equation*}
    (D\varphi_{\bG/\bH}\circ\cT_r:D\varphi_{\bG/\bH}\circ\cT_l)_t = \nu_{\bH}^{\frac 12 it^2}\delta^{it},
  \end{equation*}
  i.e. \Cref{eq:dtdt}. That these elements are actually in the relative commutant of $L^{\infty}(\bG/\bH)$ is a general feature of cocycle derivatives between operator-valued weights (\cite[Proposition 6.1]{haag2} again).

  As for \Cref{eq:sigmatdel}, it follows from \Cref{eq:dtdt} and the cocycle property of the Radon-Nikodym derivatives \cite[Proposition 6.3 (2)]{haag2}:
  \begin{equation*}
    (D\cT_r:D\cT_l)_{s+t} = (D\cT_r:D\cT_l)_s \sigma^{\cT_l}_s(D\cT_r:D\cT_l)_t.
  \end{equation*}
\end{proof}

We now have the object, alluded to in the discussion following \Cref{re:lie}, that captures the discrepancy between $\cT_r$ and $\cT_l$:

\begin{definition}\label{def:relmodel}
  Let $\bH\trianglelefteq\bG$ be a closed, normal, locally compact quantum subgroup.

  The {\it relative modular element} $\delta_{\bG\triangleright \bH}$ is the positive element affiliated with
  \begin{equation*}
    L^{\infty}(\bG/\bH)^c = L^{\infty}(\bG)\cap L^{\infty}(\bG/\bH)'
  \end{equation*}
  provided by \Cref{th:relmodel}, determined by
  \begin{equation}\label{eq:relmodel}
    (D\cT_r:D\cT_l)_t = \nu_{\bH}^{\frac 12 it^2}\delta_{\bG\triangleright\bH}^{it}. 
  \end{equation}
\end{definition}

We are now ready to generalize \Cref{th:resmod} to non-central quantum subgroups and provide the quantum counterpart to \Cref{eq:normnotcent-cls}. 

\begin{theorem}\label{th:normnotcent}
  For a closed, normal, locally compact quantum subgroup $\bH\trianglelefteq\bG$ we have
  \begin{equation}\label{eq:disintdel}
    \delta_{\bG} = \delta_{\bG\triangleright H}\delta_{\bG/\bH}. 
  \end{equation}
\end{theorem}
\begin{proof}
  Since
  \begin{itemize}
  \item $\delta_{\bG/\bH}$, which is affiliated with $L^{\infty}(\bG/\bH)$;
  \item and $\delta_{\bG\triangleright\bH}$, affiliated with the relative commutant $L^{\infty}(\bG/\bH)^c$ by \Cref{th:relmodel},
  \end{itemize}
  the two {\it strongly commute} \cite[\S 11.5]{br} in the sense that the spectral projections of one commute with those of the other. The {\it strong product} of \cite[discussion following Definition IX.2.11]{tak2} thus makes sense and is again a positive (unbounded, typically) operator; this is the meaning of the right-hand side of \Cref{eq:disintdel}.

  On the one hand, we have
  \begin{equation}\label{eq:dpsidphi}
    (D\psi_{\bG}:D\varphi_{\bG})_t = \nu_{\bG}^{\frac 12it^2}\delta_{\bG}^{it}
  \end{equation}
  by \cite[Propositoin 5.5]{vaes-rn} and \cite[Proposition 6.8 (3)]{kvcast}. On the other,
  \begin{align*}
    (D\psi_{\bG}:D\varphi_{\bG})_t &= (D\psi_{\bG/\bH}\circ \cT_r:D\varphi_{\bG/\bH}\circ \cT_l)_t\quad\text{\Cref{le:lrhaar}}\\
                                   &=(D\psi_{\bG/\bH}\circ \cT_r:D\psi_{\bG/\bH}\circ \cT_l)_t
                                     (D\psi_{\bG/\bH}\circ \cT_l:D\varphi_{\bG/\bH}\circ \cT_l)_t\quad\text{\cite[Theorem VIII.3.2]{tak2}}\\
                                   &=(D\cT_r:D\cT_l)_t
                                     (D\psi_{\bG/\bH}\circ \cT_l:D\varphi_{\bG/\bH}\circ \cT_l)_t
                                     \quad\text{\cite[Definition 6.2]{haag2}}\\
                                   &=(D\cT_r:D\cT_l)_t
                                     (D\psi_{\bG/\bH}:D\varphi_{\bG/\bH})_t
                                     \quad\text{\cite[Theorem 4.7 (2)]{haag1}}\\
                                   &=\nu_{\bH}^{\frac 12 it^2}\delta_{\bG\triangleright \bH}^{it}
                                     (D\psi_{\bG/\bH}:D\varphi_{\bG/\bH})_t
                                     \quad\text{\Cref{eq:relmodel}}\\
                                   &=\nu_{\bH}^{\frac 12 it^2}\delta_{\bG\triangleright \bH}^{it}
                                     \nu_{\bG/\bH}^{\frac 12 it^2}\delta_{\bG/\bH}^{it}
                                     \quad\text{as in \Cref{eq:dpsidphi}, applied to $\bG/\bH$}\\
                                   &=\nu_{\bG}^{\frac 12 it^2}\delta_{\bG\triangleright \bH}^{it}\delta_{\bG/\bH}^{it}\quad \text{\Cref{pr:scalemult}}.
  \end{align*}
  A comparison with \Cref{eq:dpsidphi} delivers the conclusion.
\end{proof}

\begin{remark}
  \Cref{pr:scalemult} was not, strictly speaking, necessary in the proof of \Cref{th:normnotcent}, for we could have reversed the implication as in the proof of \Cref{th:resmod}: upon obtaining the equality
  \begin{equation*}
    \nu_{\bG}^{\frac 12it^2}\delta_{\bG}^{it} =
    \nu_{\bH}^{\frac 12 it^2}
    \nu_{\bG/\bH}^{\frac 12 it^2}
    \cdot
    \delta_{\bG\triangleright \bH}^{it}
    \delta_{\bG/\bH}^{it}
  \end{equation*}
  the quadratic and linear factors automatically separate to give
  \begin{equation*}
    \nu_{\bG}^{\frac 12it^2} =
    \nu_{\bH}^{\frac 12 it^2}
    \nu_{\bG/\bH}^{\frac 12 it^2}
    \Rightarrow
    \nu_{\bG} = \nu_{\bH}\nu_{\bG/\bH}
  \end{equation*}
  (i.e. \Cref{pr:scalemult}) and the target equation \Cref{eq:disintdel}.  
\end{remark}

It will be convenient, for future reference, to collect a few assorted general remarks on relative modular elements.

\begin{proposition}\label{pr:relmodel}
  Let $\iota:\bH\trianglelefteq \bG$ be a closed, normal locally compact quantum subgroup and $\delta=\delta_{\bG\triangleleft \bH}$ the relative modular element of \Cref{def:relmodel}. The following assertions hold.
  \begin{enumerate}[(1)]
  \item\label{item:3} $\Delta_{\bG}(\delta)=\delta\otimes \delta$.    
  \item\label{item:4} $\tau_{\bG,t}(\delta)=\delta$ and $R_{\bG}(\delta)=\delta^{-1}$.
  \item\label{item:5}
    \begin{equation}\label{eq:irrel}
      L^{\infty}(\bG)\ni'\delta\stackrel{\iota_r}{\longmapsto} \delta\otimes \delta_{\bH}\in'L^{\infty}(\bG)\otimes L^{\infty}(\bH),
    \end{equation}
    where primed belonging symbols denote affiliation, per \Cref{not:inprime}.
  \item\label{item:6} similarly,
    \begin{equation}\label{eq:ilrel}
      L^{\infty}(\bG)\ni'\delta\stackrel{\iota_l}{\longmapsto} \delta_{\bH}\otimes \delta\in' L^{\infty}(\bH)\otimes L^{\infty}(\bG).
    \end{equation}
  \item\label{item:7}
    \begin{equation}\label{eq:strstl}
      \sigma^{\cT_r}_{\bG,t}(\delta) = \sigma^{\cT_l}_{\bG,t}(\delta) = \nu_{\bH}^t\delta.
    \end{equation}
  \item\label{item:8}
    \begin{equation}\label{eq:strconj}
      \sigma^{\cT_r}_{\bG,t} = \delta^{it}\sigma^{\cT_l}_{\bG,t}(\cdot)\delta^{-it}. 
    \end{equation}
  \end{enumerate}  
\end{proposition}
\begin{proof}
  Item \Cref{item:3} follows from
  \begin{itemize}
  \item the analogous statement (\cite[Proposition 7.12 (1)]{kvcast}) for the plain modular elements $\delta_{\bG}$ and $\delta_{\bG/\bH}$, which in the context of \Cref{th:normnotcent} strongly commute;
  \item together with \Cref{eq:disintdel};
  \item and the fact that the embedding
    \begin{equation}\label{eq:ghg}
      L^{\infty}(\bG/\bH)\subseteq L^{\infty}(\bG)
    \end{equation}
    intertwines the comultiplications $\Delta_{\bG/\bH}$ and $\Delta_{\bG}$. 
  \end{itemize}
  The argument is very similar for part \Cref{item:4}: analogous statements hold for $\delta_{\bG}$ and $\delta_{\bG/\bH}$ \cite[Proposition 7.12 (2)]{kvcast}, the inclusion \Cref{eq:ghg} intertwines both scaling groups and unitary antipodes \cite[Proposition A.5]{bjv}, and we can again apply \Cref{eq:disintdel}.

  To obtain \Cref{eq:irrel}, note that
  \begin{itemize}
  \item $\iota_r(\delta_{\bG})=\delta_{\bG}\otimes \delta_{\bH}$ \cite[Theorem 3.4, Corollary 3.9]{bcv};
  \item $\iota_r(\delta_{\bG/\bH})=\delta_{\bG}\otimes 1$ because $\delta_{\bG/\bH}$ is affiliated with \Cref{eq:homog};
  \item hence the conclusion, per \Cref{eq:disintdel}.
  \end{itemize}
  
  \Cref{eq:ilrel} is a consequence of \Cref{eq:irrel}, \Cref{le:rinv} and the fact that unitary antipodes turn all modular elements (absolute or relative) into their inverses (\cite[Proposition 7.12 (2)]{kvcast} and part \Cref{item:4} of this proposition).

  The last equality in \Cref{eq:strstl} is nothing but \Cref{eq:sigmatdel}, whereas the first equality will follow once we have \Cref{eq:strconj}; it thus remains to prove the latter. For that purpose, note that for
  \begin{equation*}
    a\in L^{\infty}(\bG/\bH)^c
  \end{equation*}
  we have
  \begin{align*}
    \sigma^{\cT_r}_{\bG,t}(a) &= \sigma^{\psi_{\bG/\bH}\circ\cT_r}_{\bG,t}(a)
                                \quad\text{by \cite[Definition 6.2]{haag2}}\\
                              &= \sigma^{\psi_{\bG}}_{\bG,t}(a)
                                \quad\text{\Cref{eq:phipsi}}\\
                              &=\delta_{\bG}^{it}\sigma^{\varphi_{\bG}}_{\bG,t}(a)\delta_{\bG}^{-it}
                                \quad\text{\cite[Proposition 7.12 (5)]{kvcast}}\\
                              &=\delta_{\bG}^{it}\sigma^{\varphi_{\bG/\bH}\circ\cT_l}_{\bG,t}(a)\delta_{\bG}^{-it}
                                \quad\text{\Cref{eq:phipsi} again}\\
                              &=\delta_{\bG}^{it}\sigma^{\cT_l}_{\bG,t}(a)\delta_{\bG}^{-it}
                                \quad\text{once more \cite[Definition 6.2]{haag2}}\\
                              &=\delta^{it}\delta_{\bG/\bH}^{it}\sigma^{\cT_l}_{\bG,t}(a)\delta_{\bG/\bH}^{-it}\delta^{-it}
                                \quad\text{\Cref{th:normnotcent}}\\
                              &=\delta^{it}\sigma^{\cT_l}_{\bG,t}(a)\delta^{-it}
                                \quad\text{because the middle factor is in the commutant $L^{\infty}(\bG/\bH)^c$}.
  \end{align*}
  This concludes the proof of \Cref{item:8} and the result as a whole.
\end{proof}

The block decomposition \Cref{eq:blockgen} suggests that \Cref{th:resmod} ought to generalize past central subgroups, to the case when the upper left-hand block has trivial determinant. We isolate that situation. 

\begin{proposition}\label{pr:trivact}
  For a closed, normal locally compact quantum subgroup $\iota:\bH\trianglelefteq \bG$ the two operator-valued weights
  \begin{equation*}
    \cT_l,\cT_r:L^{\infty}(\bG)\to \widehat{L^{\infty}(\bG/\bH)}_+ = \widehat{L^{\infty}(\bH\backslash \bG)}_+
  \end{equation*}
  coincide if and only if the relative modular element $\delta_{\bG\triangleright \bH}$ of \Cref{def:relmodel} is $1$.

  Furthermore, in that case $\bH$ is unimodular. 
\end{proposition}
\begin{proof}
  The two operator-valued weights coincide precisely when $(D\cT_r:D\cT_l)_t=1$ \cite[Theorem 6.5]{haag2}. That this is equivalent to
  \begin{equation*}
    \delta_{\bG\triangleright\bH}^{it}=1,\ \forall t\in \bR\Longleftrightarrow \delta=1
  \end{equation*}
  then follows from \Cref{eq:dtdt} and \Cref{eq:sigmatdel}.

  As for the unimodularity of $\bH$, it too follows from $\delta_{\bG\triangleright \bH}=1$ by \Cref{eq:irrel}.
\end{proof}

\begin{definition}\label{def:trivact}
  Let $\bH\trianglelefteq \bG$ be a closed, normal locally compact quantum subgroup. We say that {\it the conjugation (or adjoint) action of $\bG$ on $\bH$ is measure-preserving} if the equivalent conditions of \Cref{pr:trivact} hold.

  Alternative phrasing: $\bG$ {\it acts measure-preservingly (by conjugation)}.
\end{definition}

As hinted above, we have the following immediate consequence of \Cref{th:normnotcent} (and \Cref{def:trivact}):

\begin{corollary}\label{cor:reltriv}
  Let $\bH\trianglelefteq \bG$ be a closed normal quantum subgroup of an LCQG.
  
  $\bG$ acts measure-preservingly on $\bH$ if and only if $\bG$ and $\bG/\bH$ have the same modular element.  \qedhere
\end{corollary}

\begin{remark}
  Once more, the terminology of \Cref{def:trivact} is meant to convey the analogy to the classical case. To see this, let $\bH\trianglelefteq \bG$ be a closed normal subgroup of an ordinary locally compact group and denote left Haar measures by $\mu$ and as before, the classical modular function $\Delta$ by $\delta(\cdot^{-1})$.

  The disintegration formula \Cref{eq:weyl} and the relation
  \begin{equation*}
    d\mu_{\bG}(y\cdot y^{-1}) = \delta_{\bG}(y) d\mu_{\bG},\ y\in \bG
  \end{equation*}
  (and its analogue for $\bG/\bH$) easily show that
  \begin{align*}
    d\mu_{\bH}(y\cdot y^{-1}) &= \frac{\delta_{\bG}(y)}{\delta_{\bG/\bH}(y)}
                                d\mu_{\bH}\\
                              &= \delta_{\bG\triangleright \bH}(y)
                                d\mu_{\bH}
    \quad\text{\Cref{eq:disintdel}}.\\
  \end{align*}
  This delivers the classical version of \Cref{cor:reltriv}, with the phrase `acts measure-preservingly' being assigned its straightforward meaning.
\end{remark}

\section{Modular elements as morphisms}\label{se:modmor}

Classically, the inverse modular function $\delta_{\bG}$ is a continuous morphism $\bG\to (\bR_{>0},\cdot)$. This is also true in the quantum setting, for $\delta_{\bG}$, $\delta_{\bG\triangleright \bH}$, and more broadly. Echoes of these remark can be seen in \cite[Remark 5.2]{bcv} or the proof of \cite[Theorem 6.1]{dds}, though not quite stated as such. We outline the matter here with some elaboration for future reference, including one application appearing below.

Following the terminology of \cite[\S 7]{kvcast}, and by analogy with the standard phrase in use in the theory of Hopf algebras (e.g. \cite[Definition 1.3.4]{mont}):

\begin{definition}\label{def:gplike}
  Let $\bG$ be an LCQG. A strictly positive element $\delta$ affiliated with $L^{\infty}(\bG)$ or $C_0(\bG)$ or $C^u_0(\bG)$ is {\it group-like} if $\Delta(\delta)=\delta\otimes \delta$.

  In terms of bounded operators only, this is equivalent to
  \begin{equation*}
    \Delta(\delta^{it}) = \delta^{it}\otimes\delta^{it},\ \forall t\in \bR. 
  \end{equation*}
\end{definition}

The following observation merely collects together a number of ready-made results.

\begin{proposition}\label{pr:mortor}
  Let $\bG$ be a locally compact quantum group. The following sets of objects are in mutual bijection
  \begin{enumerate}[(a)]
  \item\label{item:9} morphisms $\bG\to (\bR,+)$;
  \item\label{item:10} strictly positive group-like elements affiliated with $C^u_0(\bG)$;
  \item\label{item:11} strictly positive group-like elements affiliated with $C_0(\bG)$;
  \item\label{item:12} strictly positive group-like elements affiliated with $L^{\infty}(\bG)$.
  \end{enumerate}
\end{proposition}
\begin{proof}
  Recall \Cref{not:inprime}: $\in'$ denotes the affiliation relation. $C^*$ morphisms extend to affiliated operators \cite[Theorem 1.2]{wor-aff}; we will use this implicitly in the sequel.
  
  {\bf \Cref{item:9} $\leftrightarrow$ \Cref{item:10}.} There is an comultiplication-preserving isomorphism
  \begin{equation}\label{eq:expid}
    C_0(\bR)\ni' \exp \xmapsto[\phantom{----}]{\cong} \id_{\bR_{>0}}\in' C_0(\bR_{>0})
  \end{equation}
  (dual to the usual exponential identification of the groups $(\bR,+)$ and $(\bR_{>0},\cdot)$).
  
  Because $\delta$ is strictly positive there is also a unique morphism $C_0(\bR_{>0})\rightsquigarrow C_0^u(\bG)$ sending $\id_{\bR_{>0}}$ to $\delta$ \cite[Proposition 6.5]{kus-fc}. Composing with \Cref{eq:expid} this gives, for every strictly positive group-like $\delta\in' C_0^u(\bG)$, a unique morphism
  \begin{equation*}
    \pi^u:C_0(\bR)\rightsquigarrow C_0^u(\bG)
  \end{equation*}
  sending $\exp\mapsto \delta$ and intertwining the comultiplications $\Delta_{\bR}$ and $\Delta_{\bG}$ in the sense that
  \begin{equation*}
    \Delta_{\bG}\pi^u = (\pi^u\otimes \pi^u)\Delta_{\bR}. 
  \end{equation*}
  Such a map $\pi^u$ is one of the equivalent ways of specifying a quantum-group morphism $\bG\to \bR$ \cite[Theorem 4.8]{mrw}, so we are done.  

  
  {\bf \Cref{item:10} $\leftrightarrow$ \Cref{item:11}.} Strictly positive group-likes affiliated with $C_0^u(\bG)$ project to such along the surjection $C_0^u(\bG)\to C_0(\bG)$. Conversely, they lift uniquely along the same map as explained in the proof of \cite[Proposition 10.1]{kus-univ}.

  {\bf \Cref{item:11} $\leftrightarrow$ \Cref{item:12}.} One direction is clear, $C_0(\bG)$ being contained in $L^{\infty}(\bG)$ via a coproduct-preserving inclusion. Conversely, strictly positive group-likes affiliated with $L^{\infty}(\bG)$ are in fact $C^*$-affiliated with $C_0(\bG)$, as in the proof of \cite[Proposition 7.10]{kvcast}.
 
  This concludes the proof.
\end{proof}

\begin{notation}\label{not:delunder}
  For a strictly positive group-like $\delta$ affiliated with $L^{\infty}(\bG)$ or $C_0(\bG)$ or $C^u_0(\bG)$ we write we write $\underline{\delta}$ for the corresponding morphism $\bG\to \bR$ attached to it via \Cref{pr:mortor}.
\end{notation}

\begin{corollary}\label{cor:rtinv}
  Let $\bG$ be an LCQG. Strictly positive group-like elements $\delta\in' C_0(\bG)$ are invariant under the scaling group and satisfy $R_{\bG}(\delta)=\delta^{-1}$.
\end{corollary}
\begin{proof}
  We know from \Cref{pr:mortor} that $\delta$ is the image of the canonical group-like $\exp\in' C_0(\bR)$ through a comultiplication-intertwining morphism
  \begin{equation*}
    C_0(\bR)\rightsquigarrow C_0(\bG).
  \end{equation*}
  Such morphisms also intertwine the scaling groups and unitary antipodes \cite[Remark 12.1]{kus-univ}, so it suffices to verify the claim for the universal strictly positive group-like 
  \begin{equation*}
    \delta:=\exp\in' C_0(\bR);
  \end{equation*}
  that verification is immediate, hence the conclusion.
\end{proof}

$C^*$-affiliated elements have an accompanying notion of spectrum \cite[equation (1.20)]{wor-aff}, which by \cite[discussion following Theorem 1.6]{wor-aff} specializes back to the usual concept for concrete unbounded, normal operators on Hilbert spaces (which is the situation we are concerned with here).

For positive group-likes the spectrum has some very pleasant properties.

\begin{proposition}\label{pr:posisgr}
  For an LCQG $\bG$ the strictly positive portion
  \begin{equation*}
    \mathrm{Sp}_{>0}(\delta):=\mathrm{Sp}(\delta)\setminus\{0\}
  \end{equation*}
  of the spectrum of a strictly positive group-like $\delta\in' C_0(\bG)$ is a closed subgroup of the multiplicative group $(\bR_{>0},\cdot)$.
\end{proposition}
\begin{proof}
  The spectrum of a positive (possibly unbounded) operator is a closed subset of $\bR_{\ge 0}$, hence the (topological) closure claim. It remains to argue that $\mathrm{Sp}_{>0}(\delta)$ is closed under multiplication and inversion.

  It is a simple application of the spectral theorem (e.g. \cite[Theorem 13.30]{rud-fa}) to show that for a strictly positive $T\in' B(\cH)$ 
  \begin{itemize}
  \item the positive spectrum $\mathrm{Sp}_{>0}(T^{-1})$ is 
    \begin{equation*}
      \mathrm{Sp}_{>0}(T)^{-1}:=\{t^{-1}\ |\ t\in\mathrm{Sp}_{>0}(T)\}
    \end{equation*}    
  \item and similarly, the spectrum $\mathrm{Sp}(T\otimes T)$ is the closure of 
    \begin{equation*}
      \{st\ |\ s,\ t\in\mathrm{Sp}(T)\}.
    \end{equation*}    
  \end{itemize}
  Since
  \begin{itemize}
  \item we have a morphism $\Delta_{\bG}\in\mathrm{Mor}(C_0(\bG),C_0(\bG)^{\otimes 2})$ sending $\delta$ to $\delta\otimes \delta$;
  \item and a morphism $R_{\bG}$ sending $\delta\mapsto \delta^{-1}$ by \Cref{cor:rtinv} ($R_{\bG}$ is anti-multiplicative, but this makes no difference here);
  \item for a morphism $\pi\in\mathrm{Mor}(A,B)$ the spectrum of an image $\pi(T)$ is contained in that of $T$ for any $A$-affiliated $T$ \cite[equation (1.21)]{wor-aff},
  \end{itemize}
  the conclusion follows.
\end{proof}

Let $\delta\in' C_0(\bG)$ be a strictly positive group-like. The closed subgroup
\begin{equation}\label{eq:sp>0}
  \mathrm{Sp}(\delta)_{>0}\le (\bR_{>0},\cdot)
\end{equation}
of \Cref{pr:posisgr} has an alternative interpretation as a group-theoretic invariant attached to the morphism $\underline{\delta}:\bG\to \bR$ in \Cref{not:delunder}. Every quantum-group morphism has a {\it closed image}, introduced in \cite[Definition 4.2]{kks}; we paraphrase that discussion as follows.

\begin{definition}
  Let $\pi:\bH\to \bG$ be a morphism of LCQGs. The {\it closed image} $\overline{\mathrm{im}~\pi}$ of $\pi$ is the smallest closed quantum subgroup $\iota:\overline{\mathrm{im}~\pi}\le \bG$ for which $\pi$ admits a factorization
  \begin{equation*}
    \begin{tikzpicture}[auto,baseline=(current  bounding  box.center)]
      \path[anchor=base] 
      (0,0) node (l) {$\bH$}
      +(2,.5) node (u) {$\overline{\mathrm{im}~\pi}$}
      +(4,0) node (r) {$\bG$.}
      ;
      \draw[->] (l) to[bend left=6] node[pos=.5,auto] {$\scriptstyle $} (u);
      \draw[->] (u) to[bend left=6] node[pos=.5,auto] {$\scriptstyle \iota$} (r);
      \draw[->] (l) to[bend right=6] node[pos=.5,auto,swap] {$\scriptstyle \pi$} (r);
    \end{tikzpicture}
  \end{equation*}
\end{definition}

\begin{theorem}\label{th:clim}
  Let $\bG$ be an LCQG, $\delta\in' C_0(\bG)$ a strictly positive group-like, and $\underline{\delta}:\bG\to \bR$ the morphism associated to it as in \Cref{not:delunder}.

  The closed image
  \begin{equation*}
    \overline{\mathrm{im}~\underline{\delta}}\le (\bR,+)
  \end{equation*}
  is precisely $\log~\mathrm{Sp}_{>0}(\delta)$, i.e. the image of the closed subgroup \Cref{eq:sp>0} under the logarithm isomorphism $(\bR_{>0},\cdot)\cong (\bR,+)$.
\end{theorem}
\begin{proof}
  Since $\delta$ is strictly positive, its functional calculus allows the application of the logarithm to produce a self-adjoint element $\log \delta$ \cite[Definition 7.16]{kus-fc}. The same goes for
  \begin{equation*}
    \log\exp = \id_{\bR}\in' C_0(\bR),
  \end{equation*}
  and since $\underline{\delta}$ intertwines these logarithm operations \cite[Proposition 6.17]{kus-fc} and by definition sends $\exp\mapsto \delta$, we have
  \begin{equation*}
    \underline{\delta}(\id_{\bR}) = \log\delta. 
  \end{equation*}
  In short, then, $\underline{\delta}$ is the unique \cite[Proposition 6.5]{kus-fc} morphism sending $\id_{\bR}$ to $\log\delta$.

  Write
  \begin{equation*}
    \bH:=\log~\mathrm{Sp}_{>0}(\delta)\le \bR. 
  \end{equation*}
  \cite[Result 6.16]{kus-fc} says that $\bH$ is precisely the spectrum of $\log\delta$, so that by \cite[Theorem 3.4]{kus-fc} $\underline{\delta}$ factors as
  \begin{equation*}
    \begin{tikzpicture}[auto,baseline=(current  bounding  box.center)]
      \path[anchor=base] 
      (0,0) node (l) {$C_0(\bR)$}
      +(2,.5) node (u) {$C_0(\bH)$}
      +(4,0) node (r) {$M(C_0^u(\bG))$,}
      ;
      \draw[->] (l) to[bend left=6] node[pos=.5,auto] {$\scriptstyle $} (u);
      \draw[->] (u) to[bend left=6] node[pos=.5,auto] {$\scriptstyle $} (r);
      \draw[->] (l) to[bend right=6] node[pos=.5,auto,swap] {$\scriptstyle \underline{\delta}$} (r);
    \end{tikzpicture}
  \end{equation*}
  where the top left arrow is the obvious restriction map and the top right map is one-to-one. That injectivity in particular means that there is no further factorization through any smaller quotients of $C_0(\bR)$, meaning precisely what was sought: $\bH\le \bR$ the smallest close subgroup factoring $\underline{\delta}$. 
\end{proof}

The application alluded to at the beginning of the section has to do with {\it property (T)} for LCQGs; this is a quantum version of the classical familiar concept (e.g. \cite[Definition 1.1.3]{bdv}). Early references in the quantum setting are \cite[Definition 3.1]{fim} for {\it discrete} quantum groups and, say, \cite[Definition 3.1]{cn} and \cite[\S 6]{dfsw} for the general concept. We also refer to \cite{dsv,dds} (which will be cited more heavily shortly) and {\it their} sources for further information. In brief (\cite[Definitions 2.3 and 3.1]{cn}):

\begin{definition}
  Let $\bG$ be an LCQG.
  \begin{enumerate}[(1)]
  \item Let $U\in M(C_0(\bG)\otimes K(\cH))$ be a unitary $\bG$-representation in the sense, say, of \cite[Definition 2.1]{dsv}. A net $\zeta_i\in \cH$ of unit vectors is {\it almost invariant} (also: constitutes an {\it almost-invariant vector}) if
    \begin{equation*}
      \|U(\eta\otimes\zeta_i)-\eta\otimes \zeta_i\|\to 0
    \end{equation*}
    for all $\eta\in L^2(\bG). $
  \item Similarly, having fixed $U$ again, a vector $\zeta\in \cH$ is {\it invariant} if
    \begin{equation*}
      U(\eta\otimes\zeta)=\eta\otimes \zeta,\ \forall \eta\in L^2(\bG). 
    \end{equation*}
  \item $\bG$ has {\it property (T)} if every unitary representation that has almost-invariant vectors in fact has non-zero invariant vectors.
  \end{enumerate}
\end{definition}

\Cref{th:tnodelta} below is a slight generalization of the fact that property-(T) quantum groups are unimodular. This latter result has appeared before a number of times: \cite[Proposition 3.2]{fim} proves the claim for discrete quantum groups, \cite[Theorem 6.3]{bk} handles second-countable locally compact quantum groups, and \cite[Theorem 6.1]{dds} proves the general result for arbitrary property-(T) LCQGs. These all deal with the specific group-like element $\delta_{\bG}$; among them, the first and third both bear similarities to the argument below.

\begin{theorem}\label{th:tnodelta}
  For an LCQG $\bG$ with property (T) the only strictly positive group-like $\delta\in' C_0(\bG)$ is $1$.
\end{theorem}
\begin{proof}
  Consider $\delta\in' C_0(\bG)$ as in the statement. \Cref{pr:mortor} and \Cref{th:clim} provide us with a morphism $\bG\to \bR$ whose closed image is
  \begin{equation*}
    \bH:=\{\log t\ |\ 0\ne t \in \mathrm{Sp}(\delta)\} \le (\bR,+). 
  \end{equation*}
  But the closed image of the resulting morphism $\bG\to \bH$ is then all of $\bH$ essentially by definition. It follows that $\bG\to \bH$ has {\it dense image} in the sense of \cite[Definition 2.8]{dsv} (see \cite[discussion following the statement of Theorem A.1]{dsv}), and hence $\bH$ also has property (T) by \cite[Theorem 5.7]{dsv}.

  Being classical abelian and property-(T), $\bH$ must be compact \cite[Theorem 1.1.6]{bdv} and hence trivial because it is a subgroup of $(\bR,+)$. We are now done: the spectrum of the strictly-positive operator $\delta$ is $\{1\}$, so $\delta=1$.
\end{proof}



\addcontentsline{toc}{section}{References}

\Addresses

\end{document}